\documentclass[10pt]{amsart}
\usepackage{graphicx} 

\title{{Terminal} Absoluteness of Collapse Forcings}
\author{Cesare Straffelini}
\date{May 2026. This work is part of the author's PhD thesis, written under the supervision of Joan Bagaria at the Universitat de Barcelona and the Università degli Studi di Trento.}

\usepackage{amsmath, amssymb}

\usepackage{xcolor}

\renewcommand\le\leqslant
\renewcommand\ge\geqslant

\newtheorem{theorem}{Theorem}[section]
\newtheorem{lemma}[theorem]{Lemma}
\newtheorem{fact}[theorem]{Fact}

\newtheorem{question}[theorem]{Question}
\newtheorem*{maintheorem*}{Main Theorem}

\theoremstyle{definition}
\newtheorem{definition}[theorem]{Definition}

\begin{document}

\begin{abstract}
Generic absoluteness is the phenomenon that certain truths in the set-theoretic universe remain stable under forcing expansions. A classical result by Kripke (\cite{Kripke}) asserts that every complete Boolean algebra completely embeds into a countably generated one, implying that any forcing extension can be realised inside one obtained via a collapse forcing. This observation raises a deeper question: are all forcing notions truly necessary when studying projective generic absoluteness, or does a particular class of forcing notions suffice to capture the same level of invariance? Here we show that, under suitable large cardinal hypotheses, projective generic absoluteness for collapse forcings is indeed equivalent to absoluteness for arbitrary forcings; and we discuss the necessity of these hypotheses, showing that at a low projective level the result holds in ZFC. Thus, we reveal the {terminality} of collapse forcings since they capture the full robustness of the universe under forcing extensions.
\end{abstract}

\maketitle

\section{Introduction}

Forcing, introduced by Paul Joseph Cohen (\cite{Cohen}, \cite{Cohen2}), is a set-theoretic technique originally developed to prove the independence of the Continuum Hypothesis. Since then, it has found surprisingly many applications outside set theory: to abstract algebra (e.g., the independence of the Whitehead problem by Saharon Shelah, \cite{shelah}), to geometry (e.g., the work by Jeffrey Bergfalk on additivity of strong homology, \cite{jeff}), and even to quantum mechanics (e.g., Ilijas Farah's consistency result on all automorphisms of the Calkin algebra being inner, \cite{Farah}). Nevertheless, while forcing is a powerful instrument to construct alternative mathematical universes, most mathematical practice benefits rather from knowing which statements remain invariant under forcing. The way to measure how much the set-theoretic universe is immune to forcing is called \emph{generic absoluteness}. The concept of generic absoluteness is strongly entwined with the one of \emph{universally Baire} sets of reals, which is a strong generalisation of the usual Baire property introduced by Qi Feng, Menachem Magidor and Hugh Woodin (\cite{FMW}). In Section \ref{genabs} we define formally the concepts of generic absoluteness and universally Baire sets of reals.\medskip

A particular class of forcing notions is given by \emph{collapse forcings}, denoted by $\mathrm{Coll}(\omega,\lambda)$ for $\lambda$ a cardinal, adding a bijective function from $\omega$ to $\lambda$, thereby making $\lambda$ countable in the extension. Thanks to a theorem by Saul Kripke (\cite{Kripke}), every forcing notion of size at most $\lambda$ can be completely embedded into $\mathrm{Coll}(\omega,\lambda)$; thus, collapse forcings are terminal objects in categories of forcing notions of bounded size. This behaviour is also reflected in the context of universally Baire sets of reals: if a set $A$ is $\mathrm{Coll}(\omega,\lambda)$-universally Baire, it is also $\mathbb P$-universally Baire for every forcing notion $\mathbb P$ of size at most $\lambda$ (\cite{FMW}). Since universally Baire sets of reals and generic absoluteness are closely related, this naturally raises the following question, which is the central focus of this article: does $\boldsymbol\Sigma^1_n$-$\mathrm{Coll}(\omega,\lambda)$-Absoluteness already imply $\boldsymbol\Sigma^1_n$-$\mathbb P$-Absoluteness for any forcing notion $\mathbb P$ of size at most $\lambda$? In Section \ref{collapse}, we will see that the answer to this question is affirmative for $n\le3$, but the proof method does not immediately generalise to $n\ge4$. \medskip 

We will see that an important ingredient in the proof for the case $n\ge4$ will be the existence of certain canonical objects associated to real numbers. In Section \ref{detshar}, these objects will be the so-called \emph{sharps}. Using a result by Tony Martin (\cite{Martthisthat}), we show that -- under the existence of $a^\sharp$ for every $a\subseteq\lambda$ -- the answer to the question raised in the previous paragraph is affirmative even for $n=4$. Then, we note that thanks to some particular observations we can remove the hypothesis about $a^\sharp$, obtaining a pure ZFC result.  In Section \ref{main}, we extend this approach following our previous footsteps: sharps are replaced by iterable premice and Martin's result is generalised by Itay Neeman (\cite{Neeman}), yielding an affirmative answer for $n\ge5$. Here, though, removing the additional requirement as before cannot be done.\medskip

While the $n\le3$ cases -- and, to some extent, even the case $n=4$ under the hypothesis of sharps -- are implicit in \cite{FMW}, and higher-level arguments using sharps and mice are known, the contributions of this paper are mainly removing the sharps hypothesis in the case $n=4$, giving the generalised proof in the case $n\ge5$ -- thus showing an unified behaviour across the projective hierarchy -- and giving an argument for the `terminality' of collapse forcings.\medskip

Finding a ZFC result for the case $n\ge5$ is still open, and work in progress.


\section{Generic Absoluteness}\label{genabs}

Generic absoluteness is a property of some models of set theory, and it measures how the truth of certain statements cannot be changed by going to some forcing extension. The ultimate form of generic absoluteness would be that the universe is elementarily equivalent to any of its forcing extensions, but this is inconsistent, since we are able to force incompatible statements -- like the Continuum Hypothesis and its negation. In this article, we are interested in restricting generic absoluteness to a specific class of statements, namely some level of second-order arithmetic.

\begin{definition}
    A formula in the language of set theory is said to be
    \begin{itemize}
        \item $\Sigma^1_0$ or $\Pi^1_0$ if it is arithmetic, i.e., if all quantifiers are restricted to $\omega$;
        \item $\Sigma^1_{n+1}$ if it has the form $\exists x\subseteq\omega\ \psi(x,y_1,\ldots,y_k)$ and $\psi(x,y_1,\ldots,y_k)$ is $\Pi^1_n$;
        \item $\Pi^1_{n+1}$ if it has the form $\forall x\subseteq\omega\ \psi(x,y_1,\ldots,y_k)$ and $\psi(x,y_1,\ldots,y_k)$ is $\Sigma^1_n$.
        \end{itemize}
        A $\Sigma^1_n$- (or $\Pi^1_n$-) formula where the free variables are interpreted as parameters that are arbitrary subsets of $\omega$ is said to be a $\boldsymbol\Sigma^1_n$- (or $\boldsymbol\Pi^1_n$-) sentence.
\end{definition}

Now we are ready to formalise the concept of generic absoluteness.

\begin{definition}
    Let $\mathbb P$ be a forcing notion. We say that $\Gamma$-$\mathbb P$-Absoluteness holds, where $\Gamma$ is $\Sigma^1_n$ or $\Pi^1_n$ for some $n<\omega$, if for every $\Gamma$-sentence $\varphi$ it holds
    \[\varphi \leftrightarrow \mathbb P\Vdash \varphi.\]
    Similarly, we say that $\boldsymbol\Gamma$-$\mathbb P$-Absoluteness holds, where $\boldsymbol\Gamma$ is $\boldsymbol\Sigma^1_n$ or $\boldsymbol\Pi^1_n$ for $n<\omega$, if for every $\Gamma$-formula $\varphi(x_1,\ldots, x_k)$ and any choice of $a_1\subseteq\omega$, $\ldots$, $a_k\subseteq\omega$ we have
    \[\varphi(a_1,\ldots,a_k) \leftrightarrow \mathbb P\Vdash \varphi(\check a_1,\ldots,\check a_k).\]
\end{definition}

By ``projective absoluteness'' we refer informally to absoluteness for all projective statements, i.e., statements that are $\boldsymbol\Sigma^1_n$ or $\boldsymbol\Pi^1_n$ for some $n<\omega$.\medskip

The most fundamental result about absoluteness is the following theorem, proven by Joseph Shoenfield in \cite{Shoenfield}, shortly before the discovery of forcing by Cohen. Indeed, the original statement of the theorem is not about \emph{generic} absoluteness, but instead more general about absoluteness between models of set theory. This is outside the scope of this article, we just need the following weaker formulation:

\begin{theorem}[Shoenfield]\label{Shoenfield}
    $\boldsymbol\Sigma^1_2$-$\mathbb P$-Absoluteness holds, for any forcing notion $\mathbb P$.
\end{theorem}

Notice that, since the negation of any $\Sigma^1_{n}$-formula is equivalent to a $\Pi^1_n$-formula, having $\boldsymbol\Sigma^1_n$-$\mathbb P$-Absoluteness is exactly the same as $\boldsymbol\Pi^1_n$-$\mathbb P$-Absoluteness. This is the reason why we are going to focus only about absoluteness for $\Sigma^1_n$- and $\boldsymbol\Sigma^1_n$-sentences. Now, knowing that $\boldsymbol\Sigma^1_2$- and $\boldsymbol\Pi^1_2$-sentences are absolute for any forcing notion, we can move on to considering the next level of complexity, i.e., $\Sigma^1_3$- and $\boldsymbol\Sigma^1_3$-sentences. One direction of the generic absoluteness implication is true, by the following fact:

\begin{lemma}\label{upwards}
    Assume $\boldsymbol\Sigma^1_n$-$\mathbb P$-Absoluteness holds, for some $\mathbb P$. For any $\Sigma^1_{n+1}$-formula $\varphi(x_1,\ldots,x_k)$ and any choice of parameters $a_1\subseteq\omega$, $\ldots$, $a_k\subseteq\omega$, we have
    \[\varphi(a_1,\ldots,a_k) \to \mathbb P\Vdash \varphi(\check a_1,\ldots,\check a_k).\]
\end{lemma}

\begin{proof}
    Since $\varphi(x_1,\ldots,x_n)$ is a $\Sigma^1_{n+1}$-formula, it is of the form \[\exists y\subseteq\omega\ \psi(y,x_1,\ldots,x_k)\] for some $\Pi^1_n$-formula $\psi(y,x_1,\ldots,x_k)$. If $\varphi(a_1,\ldots,a_k)$ holds, there is a witness $w\subseteq\omega$ such that $\psi(w,a_1,\ldots,a_k)$ holds. By $\boldsymbol\Sigma^1_n$-$\mathbb P$-Absoluteness, it implies
    \[\mathbb P\Vdash \psi(\check w, \check a_1,\ldots,\check a_k)\]
    hence $\mathbb P\Vdash \exists y\subseteq\omega\ \psi(y,\check a_1,\ldots,\check a_k)\equiv \varphi(\check a_1,\ldots,\check a_k)$, as desired.
\end{proof}

This lemma, in conjunction with Shoenfield's Theorem \ref{Shoenfield}, implies as a corollary that $\boldsymbol\Sigma^1_3$-sentences are \emph{upwards} absolute for forcing extensions -- and, symmetrically, that $\boldsymbol\Pi^1_3$-sentences are {downwards} absolute. However, this represents essentially the maximum amount of generic absoluteness provable in ZFC; even $\Sigma^1_3$-sentences (i.e., without parameters) can consistently fail to be absolute even for generic extensions of the simplest non-trivial forcing notion, namely Cohen forcing ($\mathbb C$):

\begin{lemma}
    It is consistent, relative to $\mathrm{ZFC}$, that $\Sigma^1_3$-$\mathbb C$-Absoluteness fails.
\end{lemma}

\begin{proof}
    We show that, in Gödel's constructible universe $\mathsf L$, $\Sigma^1_3$-$\mathbb C$-Absoluteness fails. For the sentence ``every real is constructible'' is equivalent to a $\Pi^1_3$-sentence $\varphi$, and while it is true that $\mathsf L\models \varphi$, we have that $\mathsf L[c]\models \lnot\varphi$, for any $c\subseteq \mathbb C$ that is $\mathsf L$-generic, because $c$ itself is a non-constructible real in $\mathsf L[c]$.
\end{proof}

In the paper \cite{BagFri}, it is shown that $\Sigma^1_3$-$\mathbb P$-Absoluteness for every forcing notion $\mathbb P$ is equiconsistent with ZFC, but already $\boldsymbol\Sigma^1_3$-$\mathbb P$-Absoluteness for every forcing notion $\mathbb P$ has some large cardinal strength: precisely, it is equiconsistent with the existence of a $\Sigma_2$-reflecting cardinal. This was proved, independently, also in \cite{FMW}.\medskip

Climbing the projective hierarchy, the full schema of $\boldsymbol\Sigma^1_n$-$\mathbb P$-Absoluteness for every $n<\omega$ and every forcing notion $\mathbb P$ is equiconsistent with the existence of infinitely-many strong cardinals. More precisely, a result attributed to Woodin (to this date still unpublished, but found for example in \cite{Wilson}) shows that if $n$ strong cardinals exist, then in a collapse forcing extension $\boldsymbol\Sigma^1_{n+3}$-$\mathbb P$-Absoluteness holds for every $\mathbb P$. Later, Kai Hauser (\cite{Kai}) proved the converse consistency implication.\medskip

There is a strong intertwining between generic absoluteness and universal Baireness for sets of reals (introduced in \cite{FMW}). There are two equivalent definitions of this property, but here we focus on the one where the connection with generic absoluteness is more evident. Recall that, as is customary in set theory, when we talk about the set of real numbers we typically mean either the powerset $\wp(\omega)$ of $\omega$ or the set ${}^\omega\omega$ of sequences of length $\omega$ with values in $\omega$: we do not concern ourselves with whether a real is coded by a set or by a sequence of naturals.

\begin{definition}
    Let $A\subseteq{}^\omega\omega$ and let $\mathbb P$ be a forcing notion. We say that the set $A$ is $\mathbb P$-universally Baire ($\mathbb P$-uB for short) if there are two trees $S\subseteq{}^{<\omega}\omega\times {}^{<\omega}\zeta$ and $T\subseteq{}^{<\omega}\omega\times{}^{<\omega}\zeta$, for $\zeta$ an ordinal, with the following properties:
    \begin{itemize}
        \item $A=\pi_1''[S]$;
        \item ${}^\omega\omega-A=\pi_1''[T]$;
        \item $\mathbb P\Vdash \pi_1''[S]\cup\pi_1''[T]={}^\omega\omega$. 
    \end{itemize}
\end{definition}

To clarify the notation: if $Q$ is any tree, by $[Q]$ we mean the set of branches of $Q$. From the definition, it follows that $[S]\subseteq {}^\omega\omega\times{}^\omega\zeta$ and $[T]\subseteq {}^\omega\omega\times{}^\omega\zeta$. Moreover, the map $\pi_1:{}^\omega\omega\times{}^\omega\zeta\to{}^\omega\omega$ is the first projection, and also, if $f$ is a map and $a$ a set, $f''a$ denotes the image of $a$ under $f$, i.e., $f''a=\{f(x):x\in a\}$.\medskip

The idea behind this definition is to have a pair $\langle S,T\rangle$ of trees projecting to $A$ and to its complement, with the property that even in a $\mathbb P$-generic extension the pair $\langle  S, T\rangle$ still projects to complements. Indeed, in the definition we ask $\mathbb P$ only to force that $\pi_1''[S]$ and $\pi_1''[T]$ together cover the whole set of reals (so, a priori, we do not require their intersection to be empty), but it is easy to see the following:
\begin{lemma}
    Let $S\subseteq{}^{<\omega} \omega\times{}^{<\omega}\zeta$ and $T\subseteq{}^{<\omega}\omega \times{}^{<\omega}\zeta$ be trees and assume that \[\pi_1''[S]\cap\pi_1''[T]=\varnothing.\] Then, for any forcing notion $\mathbb P$, it holds that
    \[\mathbb P\Vdash \pi_1''[S]\cap\pi_1''[T]=\varnothing.\]
\end{lemma}

\begin{proof}
Let $G\subseteq\mathbb P$ be $\mathsf V$-generic and assume, towards a contradiction, that
\[\pi_1''[S]^{\mathsf V[G]}\cap \pi_1''[T]^{\mathsf V[G]}\neq\varnothing.\]
Let $Z$ be the set defined by
\[Z:=\{\langle s,t_1,t_2\rangle: \langle s,t_1\rangle \in S,\ \langle s,t_2\rangle\in T\}.\]
In $\mathsf V[G]$, we have $[Z]\neq\varnothing$, hence $\langle Z,\supseteq\rangle$ is ill-founded. Since well-foundedness is absolute, $\langle Z,\supseteq\rangle$ is ill-founded also in $\mathsf V$, so we have $[Z]\neq\varnothing$ in $\mathsf V$, meaning that there is $\langle x,y_1,y_2\rangle\in [Z]\cap\mathsf V$, but then $x\in \pi_1''[S]\cap\pi_1''[T]$, a contradiction.
\end{proof}

From the definition, it is clear that there is some relation between universally Baire sets of reals and generic absoluteness. Recall the following definition:

\begin{definition}
    A set $A$ of reals is said to be
    \begin{itemize}
        \item $\Sigma^1_n$ if there is a $\Sigma^1_n$-formula $\psi(x)$ such that $a\in A\leftrightarrow \psi(a)$;
        \item $\Pi^1_n$ if there is a $\Pi^1_n$-formula $\psi(x)$ such that $a\in A\leftrightarrow \psi(a)$;
        \item $\Delta^1_n$ if it is both $\Sigma^1_n$ and $\Pi^1_n$;
        \item $\boldsymbol\Sigma^1_n$ if there is a $\Sigma^1_n$-formula $\psi(x,y_1,\ldots,y_k)$ and $r_1\subseteq\omega$, $\ldots$, $r_k\subseteq\omega$ with
        \[a\in A\leftrightarrow \psi(a,r_1,\ldots,r_k);\]
        \item $\boldsymbol\Pi^1_n$ if there is a $\Pi^1_n$-formula $\psi(x,y_1,\ldots,y_k)$ and $r_1\subseteq\omega$, $\ldots$, $r_k\subseteq\omega$ with
        \[a\in A\leftrightarrow \psi(a,r_1,\ldots,r_k);\]
        \item $\boldsymbol\Delta^1_n$ if it is both $\boldsymbol\Sigma^1_n$ and $\boldsymbol\Pi^1_n$.
    \end{itemize}
\end{definition}

\begin{lemma}\label{uBabs}
    Assume that every $\boldsymbol\Sigma^1_n$-set of reals is $\mathbb P$-universally Baire, for some forcing notion $\mathbb P$. Then, $\boldsymbol\Sigma^1_{n+1}$-$\mathbb P$-Absoluteness holds.
\end{lemma}

\begin{proof}
By induction on $n<\omega$. The case $n=1$ follows from Shoenfield's Theorem \ref{Shoenfield}. For the inductive step, we may already assume that $\boldsymbol\Sigma^1_n$-$\mathbb P$-Absoluteness holds.\medskip

Let $\varphi(x_1,\ldots,x_k)$ be any $\Sigma^1_{n+1}$-formula and $a_1\subseteq\omega$, $\ldots$, $a_k\subseteq\omega$. By Lemma \ref{upwards}, it is enough to show downwards absoluteness for $\varphi$. Hence, we can assume that
\[\mathbb P\Vdash \exists y\subseteq\omega\ \psi(y,a_1,\ldots,a_k)\hspace{.5cm}(\heartsuit)\]
where $\varphi\equiv \exists y\subseteq\omega\ \psi(y,x_1,\ldots,x_k)$ for a $\Pi^1_n$-formula $\psi(y,x_1,\ldots,x_k)$. Now let
    \[A:=\{y\subseteq\omega: \psi(y,a_1,\ldots,a_k)\}.\]
    By definition, $A$ is a $\boldsymbol\Pi^1_n$-set, so ${}^\omega\omega-A$ is $\mathbb P$-universally Baire. Let $T\subseteq{}^{<\omega}\omega\times{}^{<\omega}\zeta$ be a tree with $\pi_1''[T]=A$ as in the definition of the $\mathbb P$-universally Baire property. From $(\heartsuit)$, we get that for every $\mathsf V$-generic $G\subseteq\mathbb P$ it holds $\pi_1''[T]^{\mathsf V[G]}\neq\varnothing$, that is equivalent to $[T]^{\mathsf V[G]}\neq\varnothing$ hence also to $\langle T,\supseteq\rangle$ being ill-founded in $\mathsf V[G]$. Since well-foundedness is absolute, we also get that $[T]^{\mathsf V}\neq\varnothing$ hence $A\neq\varnothing$ so $\varphi(a_1,\ldots,a_k)$ holds, implying downwards absoluteness for $\varphi$, as desired.
\end{proof}

\section{Collapse Forcings}\label{collapse}

The collapse forcing notion $\mathrm{Coll}(\omega,\lambda)$, for an infinite cardinal $\lambda$, adds a bijection between $\lambda$ and $\omega$, thus making $\lambda$ countable in the forcing extension. Formally, the collapse forcing is defined as follows.

\begin{definition}
    Let $\lambda$ be an ordinal. We let $\mathrm{Coll}(\omega,\lambda)$ be the poset consisting of all finite partial functions $p:\omega\to\lambda$, ordered by reverse inclusion.
\end{definition}

We refer to \cite[Section 15]{Jech} for more information and facts about the collapse forcing. An important remark is that in a certain sense we can think of the collapse forcing as the \emph{terminal object} in the category of forcing notions of size at most $\lambda$. This is a corollary of a well-known result of Saul Kripke (\cite{Kripke}).

\begin{theorem}[Kripke]\label{terminal}
    Let $\mathbb P$ be any forcing notion, $|\mathbb P|\le\lambda$. Then
    \[\mathbb P\times \mathrm{Coll}(\omega,\lambda)\equiv \mathbb P\ast \mathrm{Coll}(\omega, \lambda)\equiv \mathrm{Coll}(\omega,\lambda),\]
    where $\equiv$ denotes forcing-equivalence.
\end{theorem}

 It is noted also in \cite{FMW}, where the authors exploit it to show the following:

\begin{lemma}\label{uBdown}
    Let $A$ be a set of reals that is $\mathrm{Coll}(\omega,\lambda)$-universally Baire, for $\lambda$ a cardinal. Then $A$ is $\mathbb P$-universally Baire for any forcing notion $\mathbb P$ with $|\mathbb P|\le\lambda$.
\end{lemma}

At this point, having seen the relation between universally Baire sets and generic absoluteness, and the fact that collapse forcing suffices to determine the universally Baire property for all forcing posets of no greater size, it is natural to ask:

\begin{question}\label{lukas}
    Assume $\boldsymbol\Sigma^1_n$-$\mathrm{Coll}(\omega,\lambda)$-Absoluteness, for $\lambda$ a cardinal and $n<\omega$. Does $\boldsymbol\Sigma^1_n$-$\mathbb P$-Absoluteness hold for every $\mathbb P$ with $|\mathbb P|\le\lambda$?
\end{question}

Clearly, Question \ref{lukas} has a positive answer for $n=2$ and any cardinal $\lambda$, by Shoenfield's Theorem \ref{Shoenfield}. We can also immediately notice that the answer is positive in the case $n=3$, a fact already mentioned in the article \cite{FMW}:

\begin{lemma}\label{n=3}
We have that  $\boldsymbol\Sigma^1_{3}\text{-}\mathrm{Coll}(\omega,\lambda)\text{-Absoluteness} \Rightarrow \boldsymbol\Sigma^1_{3}\text{-}\mathbb P\text{-Absoluteness}$, for $\lambda$ a cardinal and $\mathbb P$ a forcing notion with $|\mathbb P|\le\lambda$.
\end{lemma}

\begin{proof}
Let $g\subseteq\mathbb P$ be $\mathsf V$-generic, $H\subseteq\mathrm{Coll}(\omega,\lambda)$ be $\mathsf V[g]$-generic and $G\subseteq\mathrm{Coll}(\omega,\lambda)$ be $\mathsf V$-generic with $\mathsf V[g][H]=\mathsf V[G]$, thanks to Kripke's Theorem \ref{terminal}.\medskip

By Shoenfield's Theorem \ref{Shoenfield} and Lemma \ref{upwards}, we just need to show downwards absoluteness from $\mathsf V[g]$ to $\mathsf V$.  Hence, assume $\varphi$ is a $\boldsymbol\Sigma^1_3$-sentence true in $\mathsf V[g]$, with parameters in $\mathsf V$. Again thanks to Shoenfield's Theorem \ref{Shoenfield} and Lemma \ref{upwards} we know that $\varphi$ holds also in $\mathsf V[g][H]=\mathsf V[G]$. But now, by $\boldsymbol\Sigma^1_3$-$\mathrm{Coll}(\omega,\lambda)$-Absoluteness in $\mathsf V$, since $\varphi$ has parameters in $\mathsf V$ we get that it is true in $\mathsf V$, as desired.
\end{proof}

Looking at the proof of Lemma \ref{n=3}, it would be tempting to give a general positive answer to Question \ref{lukas} by induction on $n$, by following the same proof strategy. But this already fails in the $n=4$ case. Indeed, in order to get upwards absoluteness for $\boldsymbol\Sigma^1_4$-sentences from $\mathsf V[g]$ to $\mathsf V[g][H]$ we would need $\boldsymbol\Sigma^1_3$-$\mathrm{Coll}(\omega,\lambda)$-Absoluteness to hold in $\mathsf V[g]$, but a priori there is no reason why this should hold. If working with $\boldsymbol\Sigma^1_3$-sentences having parameters in $\mathsf V$, we could exploit the absoluteness between $\mathsf V$ and $\mathsf V[g]$ and between $\mathsf V$ and $\mathsf V[G]$ to show they are also absolute between $\mathsf V[g]$ and $\mathsf V[G]=\mathsf V[g][H]$, when we instead allow the parameters to range over the whole set of real numbers of $\mathsf V[g]$ there is no immediate way to let them ``pass through'' $\mathsf V$, since they may be represented by a $\mathbb P$-name too big to be coded by a real.

\section{Sharps}\label{detshar}

The solution to the problem discussed at the end of the previous section is already solved in \cite{FMW}, where they prove the following result:

\begin{lemma}\label{overkill}    Assume $a^\sharp$ exists for every set $a$ of ordinals. Then, in any generic extension of $\mathsf V$, $\boldsymbol\Sigma^1_3$-$\mathbb P$-Absoluteness holds, for any (set-)forcing notion $\mathbb P$.
\end{lemma}

If we assume the additional hypothesis of $a^\sharp$ existing for any set $a$ of ordinals, we get that $\boldsymbol\Sigma^1_4$-$\mathrm{Coll}(\omega,\lambda)$-Absoluteness implies $\boldsymbol\Sigma^1_4$-$\mathbb P$-Absoluteness, for any cardinal $\lambda$ and any $|\mathbb P|\le\lambda$. Though, this result is not optimal. Indeed, the goal of this section is to remove this additional hypothesis, proving that $\boldsymbol\Sigma^1_4$-$\mathrm{Coll}(\omega,\lambda)$-Absoluteness already implies the existence of all the sharps that are really needed in the proof of the result. We recall briefly the definition of a sharp. For more information, see \cite[Section 9]{Kanamori} or \cite{SoloNon}.

\begin{fact}\label{sharpsdef}
    The following are equivalent, for $a$ a set of ordinals:
    \begin{itemize}
        \item there is a non-trivial embedding $j:\mathsf L[a]\hookrightarrow\mathsf L[a]$ with $\mathrm{crit}(j)>\sup(a)$;
        \item there is a club class $\mathcal I_a$ of Silver indiscernibles for $\mathsf L[a]$;
        \item there is an uncountable set $X$ of ordinals such that there is no set $Y$ with $X\subseteq Y$, $|X|=|Y|$ and $Y\in\mathsf L[a]$.
    \end{itemize}
    In all these cases, we say that $``a^\sharp$ exists". The set $a^\sharp$ is meant to be a set of ordinals coding the formul\ae\ with $n$ free variables satisfied in $\mathsf L[a]$ by some (or, equivalently, all) increasing $n$-uples of Silver indiscernibles in $\mathcal I_a$, for all $n<\omega$.
\end{fact}

Note that, if $a^\sharp$ exists, all sufficiently large cardinals above $\sup(a)$ are Silver indiscernibles. One important fact about sharps for us is the following well-known statement, saying that sharps for sets in the ground model cannot be added by forcing. The author thanks Andreas Blass for communicating a proof of this folklore result, which appears to be difficult to locate in the literature.

\begin{lemma}\label{abs1}
Let $a$ be a set of ordinals, and assume
\[\mathbb P\Vdash ``a^\sharp\ \text{exists}"\]
for some forcing poset $\mathbb P$. Then, $a^\sharp$ exists.
\end{lemma}

\begin{proof}
Since $\mathbb P\Vdash ``a^\sharp\ \text{exists}"$, there is a $\mathbb P$-name $\tau$ such that $\mathbb P\Vdash \tau=a^\sharp$. Towards a contradiction, suppose that there is an ordinal $\xi$ and there are two conditions $p$ and $q$ in $\mathbb P$ such that $p\Vdash \xi\in \tau$ while $q\Vdash \xi\not\in \tau$.\medskip

Let $\varphi(x_1,\ldots,x_n)$ be the formula coded by the ordinal $\xi$, and let $\kappa_1<\cdots<\kappa_n$ be cardinals above $\sup(a)$ and $|\mathbb P|$, which are Silver indiscernibles. Since these are above $|\mathbb P|$, they remain cardinals, thus Silver indiscernibles, after forcing with $\mathbb P$. So, we have that $p\Vdash (\mathsf L[a]\models \varphi(\kappa_1,\ldots,\kappa_n))$ but $q\Vdash(\mathsf L[a]\models \lnot\varphi(\kappa_1,\ldots,\kappa_n))$. This is impossible, since the $\mathsf L[a]$ of the forcing extension is the same $\mathsf L[a]$ of the ground model, and $a$, $\kappa_1$, $\ldots$, $\kappa_n$ all belong to $\mathsf V$.\medskip

So, we conclude that no two conditions can disagree on membership of an ordinal $\xi$ in the $\mathbb P$-name $\tau$. Since the conditions that decide whether $\xi\in \tau$ are dense in $\mathbb P$, it follows that for each ordinal $\xi$ either $\mathbb P\Vdash \xi\in \tau$ or $\mathbb P\Vdash \xi\not\in\tau$. Let
\[b:=\{\xi\in\mathsf{Ord}: \mathbb P\Vdash \xi\in\tau\}.\]
This set $b$ is in the ground model and encodes the formul\ae\ satisfied in $\mathsf L[a]$ by some Silver indiscernibles in the forcing extension. But since $a$ is in the ground model, and the Silver indiscernibles can be picked as before to be in the ground model, $c$ is coding the formul\ae\ satisfied in $\mathsf L[a]$ by Silver indiscernibles, i.e., this implies the existence of $a^\sharp$ in $\mathsf V$ (we cannot say $c=a^\sharp$ since the coding could be different).
\end{proof}

Another results concerning sharps and absoluteness is the following, that instead is about lifting upwards sharps already present in the ground model:

\begin{lemma}\label{abs2}
    Assume $a^\sharp$ exists for all $a\subseteq\lambda$. Let $\mathbb P$ be such that $|\mathbb P|\le\lambda$. Then,
    \[\mathbb P \Vdash \forall a\subseteq\lambda\ ``a^\sharp\ \text{exists}".\]
\end{lemma}

\begin{proof}
Let $\tau$ be a $\mathbb P$-name for a subset of $\lambda$, and let $G\subseteq\mathbb P$ be $\mathsf V$-generic. We want
\[\mathsf V[G]\models ``\imath_G(\tau)^\sharp\text{ exists}".\]
Since $|\mathbb P|\le\lambda$, we know that $\tau$ can be coded into a subset of $\lambda$ in $\mathsf V$, so in particular $\tau^\sharp$ exists. Also, again from $|\mathbb P|\le\lambda$, we have that $\mathbb P^\sharp$ exists, thus $\langle \mathbb P,\tau\rangle^\sharp$ exists.\medskip

Hence, let $j:\mathsf L[\mathbb P,\tau]\hookrightarrow \mathsf L[\mathbb P,\tau]$ be a non-trivial elementary embedding. Since $G$ is generic, we can lift $j$ to some non-trivial elementary embedding (in $\mathsf V[G]$) \[j':\mathsf L[\mathbb P,\tau][G]\hookrightarrow \mathsf L[\mathbb P,\tau][G].\]
Now, $\imath_G(\tau)\in \mathsf L[\mathbb P,\tau][G]$, thus if we restrict $j'$ to $\mathsf L[\imath_G(\tau)]$ we obtain
\[j'\upharpoonright \mathsf L[\imath_G(\tau)]: \mathsf L[\imath_G(\tau)]\hookrightarrow \mathsf L[\imath_G(\tau)],\]
thus ensuring us that $\imath_G(\tau)^\sharp$ exists in $\mathsf V[G]$, as desired.
\end{proof}

Recall also the following important result, from \cite{Martthisthat}:

\begin{theorem}[Martin]\label{martincountable}
Assume $a^\sharp$ exists for all $a\subseteq\omega$, and let $\psi$ be a $\Pi^1_2$-formula. Then, there are a formula $\varphi$ and $k<\omega$ such that, for each $x\subseteq\omega$, each $u\subseteq\omega$ such that $x\in \mathsf L[u]$, and each choice of indiscernibles $c_0<\cdots<c_{k-1}$ for $\mathsf L[u]$,
\[\psi(x) \Leftrightarrow \mathsf L[u]\models \varphi(x,c_0,\ldots,c_{k-1}).\]
\end{theorem}

We can immediately apply Martin's Theorem \ref{martincountable} to find that

\begin{theorem}\label{martincons}
If $a^\sharp$ exists for all $a\subseteq\lambda$, then $ \boldsymbol\Sigma^1_3\text{-}\mathrm{Coll}(\omega,\lambda)\text{-Absoluteness}$ holds.
\end{theorem}

\begin{proof}
Let $G\subseteq\mathrm{Coll}(\omega,\lambda)$ be $\mathsf V$-generic. Thanks to Lemma \ref{upwards} and to Shoenfield's Theorem \ref{Shoenfield}, we only need to prove downward absoluteness for $\boldsymbol\Sigma^1_3$-statements.\medskip

Thus, let $\psi$ be a $\Pi^1_2$-formula and assume that
\[\mathsf V[G]\models \exists x\subseteq\omega\ \psi(x,b)\]
where $b\in\mathsf V$ is a real parameter. Let $w\in\mathsf V[G]$ be a witness for the formula in $\mathsf V[G]$. Our goal is to show that such a witness also exists in $\mathsf V$.\medskip

Let $\tau$ be a $\mathrm{Coll}(\omega,\lambda)$-name for $w$. Since $\tau$ can be coded into a subset of $\lambda$ in $\mathsf V$, in particular $\tau^\sharp$ exists. Since $b$ is a real, $b^\sharp$ also exists, thus $\langle b,\tau\rangle^\sharp$ exists.\medskip

Let $\varphi$ and $k<\omega$ be the ones prescribed by Martin's Theorem \ref{martincountable} for $\psi$. We are going to assume that $k=0$; the case with more indiscernibles is analogous, but more cumbersome in writing. By hypothesis, since $w\in \mathsf L[b,\tau][G]$,
\[\mathsf L[b,\tau][G] \models \varphi(w,b).\]
This, in turn, implies that (since $\mathrm{Coll}(\omega,\lambda)\in\mathsf L[b,\tau]$)
\[\mathsf L[b,\tau] \models \mathrm{Coll}(\omega,\lambda)\Vdash \varphi(\tau,b).\]
Now, let $\mathsf N$ be a countable elementary substructure of $\mathsf L[b,\tau]$ such that $b\in\mathsf N$, $\tau\in \mathsf N$ and $\lambda\in\mathsf N$. Let $\mathsf M$ be the Mostowski collapse of $\mathsf N$, and $\pi:\mathsf N\to\mathsf M$ be the collapse map. Notice that $\pi(b)=b$ since $b\subseteq\omega$ is a real. Then, by elementarity,
\[\mathsf M\models \mathrm{Coll}(\omega,\pi(\lambda))\Vdash \varphi(\pi(\tau),b).\]
Since $\mathsf M$ is countable, in $\mathsf V$ there is a $\mathrm{Coll}(\omega,\pi(\lambda))$-generic filter $H$ over $\mathsf M$. Thus, \[\mathsf M[H]\models \varphi(\imath_G(\pi(\tau)),b).\]
Notice now that $\mathsf L[b,\tau]$ is a model of the sentence ``there is $u$ such that $\mathsf V=\mathsf L[u]$''. This is expressible as a first-order statement, hence $\mathsf M$ is also a model of it. So, we can apply Martin's Theorem \ref{martincountable} again, finding that
\[\mathsf V\models \psi(\imath_G(\pi(\tau)),b),\]
hence we found a witness $w':=\imath_G(\pi(\tau))\in\mathsf V$, as desired.
\end{proof}

Now, we are able to give a positive answer to Question \ref{lukas} for $n=4$. Let us start by proving it under the assumption that $a^\sharp$ exists for all $a\subseteq\lambda$, then we will explain how to get rid of that hypothesis.

\begin{theorem}\label{main-1}
    Let $\lambda$ be a cardinal and $\mathbb P$ a forcing notion with $|\mathbb P|\le\lambda$. Assume that $a^\sharp$ exists for every set $a\subseteq\lambda$ of ordinals. Then,
    \[\boldsymbol\Sigma^1_{4}\text{-}\mathrm{Coll}(\omega,\lambda)\text{-Absoluteness} \Rightarrow \boldsymbol\Sigma^1_{4}\text{-}\mathbb P\text{-Absoluteness}.\]
\end{theorem}

\begin{proof}
    Let $g\subseteq\mathbb P$ be $\mathsf V$-generic, $H\subseteq\mathrm{Coll}(\omega,\lambda)$ be $\mathsf V[g]$-generic, and $G\subseteq\mathrm{Coll}(\omega,\lambda)$ be $\mathsf V$-generic and such that $\mathsf V[g][H]=\mathsf V[G]$, thanks to Kripke's Theorem \ref{terminal}.\medskip

Now notice that, thanks to Lemmata \ref{n=3} and \ref{upwards}, we only need to show downward absoluteness of $\boldsymbol\Sigma^1_4$-sentences from $\mathsf V[g]$ to $\mathsf V$. By Lemma \ref{abs2}, we also have
\[\mathsf V[g] \models \forall a\subseteq\lambda\ ``a^\sharp\ \text{exists}".\]
Thus, by Theorem \ref{martincons}, we know that \[\mathsf V[g]\models\boldsymbol\Sigma^1_3\text{-}\mathrm{Coll}(\omega,\lambda)\text{-Absoluteness.}\]
If $\varphi$ is a $\boldsymbol\Sigma^1_4$-sentence true in $\mathsf V[g]$ with parameters in $\mathsf V$, by the fact that $\mathsf V[g]$ satisfies $\boldsymbol\Sigma^1_3$-$\mathrm{Coll}(\omega,\lambda)$-Absoluteness and from Lemma \ref{upwards}, we get $\mathsf V[g][H]=\mathsf V[G]\models\varphi$. But $\mathsf V$ satisfies $\boldsymbol\Sigma^1_4$-$\mathrm{Coll}(\omega,\lambda)$-Absoluteness, hence $\mathsf V\models\varphi$, as desired.
\end{proof}

Now, we want to show the following:

\begin{theorem}\label{n=4}
    Let $\lambda$ be a regular cardinal and assume $\boldsymbol\Sigma^1_4$-$\mathrm{Coll}(\omega,\lambda)$-Absoluteness holds. Then, $a^\sharp$ exists for every set of ordinals $a\subseteq\lambda$.
\end{theorem}

\begin{proof}
  First we want to show that, for all $a\subseteq\omega$, the sharp $a^\sharp$ exists. This is done in a similar way as in \cite[Lemma 3.13]{Kai}: towards a contradiction, assume $a^\sharp$ does not exist, for some real $a\subseteq\omega$. By Jensen's covering lemma for $\mathsf L[a]$, we have
\[\mathrm{Coll}(\omega,\lambda)\Vdash \exists x\subseteq\omega\ (x\ \text{codes a countable ordinal}\ \xi_x\ \text{with}\ \omega_1=(\xi_x^+)^{\mathsf L[a]}).\] 

Now, thanks to $\boldsymbol\Sigma^1_4$-$\mathrm{Coll}(\omega,\lambda)$-Absoluteness, we can find some $y\subseteq\omega$ in the ground model such that $y$ codes a countable ordinal $\xi_y$ with $\omega_1=(\xi_y^+)^{\mathsf L[a]}$. Now,
\[\mathrm{Coll}(\omega,\lambda)\Vdash \lambda^+=\omega_1 > \omega_1^{\mathsf V}=(\xi_y^+)^{\mathsf L[a]},\]
thus by $\boldsymbol\Sigma^1_4$-$\mathrm{Coll}(\omega,\lambda)$-Absoluteness again we would get that $\omega_1>(\xi_y^+)^{\mathsf L[a]}$, a contradiction. So this shows that $a^\sharp$ exists for every real $a\subseteq\omega$.\medskip

Now, the statement ``$a^\sharp$ exists for all $a\subseteq\omega$" is $\Sigma^1_4$, so it is absolute. Thus,
\[\mathrm{Coll}(\omega,\lambda)\Vdash ``a^\sharp\ \text{exists for all }a\subseteq\omega".\]
If now $a\subseteq\lambda$ is a set of ordinals in $\mathsf V$, since the cardinal $\lambda$ is collapsed to be countable by $\mathrm{Coll}(\omega,\lambda)$, the forcing notion $\mathrm{Coll}(\omega,\lambda)$ also forces $a^\sharp$ to exist. By Lemma \ref{abs1}, we now deduce that $\mathsf V\models \forall a\subseteq\lambda \ ``a^\sharp\ \text{exists}"$, as desired.
\end{proof}

\section{General Result}\label{main}

We are now ready to prove the result in the case $n\ge5$, under suitable hypotheses. In order to do so, we are going to need the following definitions, that generalise the concept of sharps.

\begin{definition}
    Let $a$ be a set of ordinals and $n>0$. Consider the smallest inner model $\mathsf W$ of $\mathrm{ZFC}$ with $a\in\mathsf W$ and $\mathsf W\models``\text{there are } n\text{-many Woodin cardinals}"$, that is also $(\omega_1+1)$-iterable. If such a model exist, we call it $\mathsf M_n(a)$. Notice that \[\mathsf M_n(a)\models ``\mathsf V=\mathsf L[\mathsf V_{\delta_n(a)}]",\] where $\delta_n(a)$ is the biggest Woodin cardinal of $\mathsf M_n(a)$.
\end{definition}

This model $\mathsf M_n(a)$ takes the place of $\mathsf L[a]$ in Fact \ref{sharpsdef}, in order to define the \emph{mouse} $M_n^\sharp(a)$. As for $a^\sharp$, we do not really define the meaning of $M_n^\sharp(a)$, but we characterise its existence via an equivalent characterisation. The original definition and proof of the equivalence can be found for example in \cite[Section 5.1]{Schimmer}.

\begin{definition}
     Given a set of ordinals $a$, we say $M_n^\sharp(a)$ exists if and only if $\mathsf M_n(a)$ exists and there is a club class $\mathcal I^n_a$ of indiscernibles for the model $\mathsf M_n(a)$.
\end{definition}

Lemma \ref{abs2} generalises to this new context, with an almost identical proof:

\begin{lemma}\label{abs4}
Assume $M_n^\sharp(a)$ exists for all $a\subseteq\lambda$, for $n>0$. Let  $|\mathbb P|\le\lambda$. Then,
\[\mathbb P \Vdash \forall a\subseteq\lambda\ ``M_n^\sharp(a)\ \text{exists}".\]
\end{lemma}

Martin's Theorem \ref{martincountable} is then replaced by the following result by Neeman (\cite{Neeman}):

\begin{theorem}[Neeman]\label{neeeeeeman}
Assume $M_n^\sharp(a)$ exists for all $a\subseteq\omega$, for $n>0$, and let $\psi$ be a $\Pi^1_{n+2}$-formula. Then, there are a formula $\varphi$ and $k<\omega$ such that, for each $x\subseteq\omega$, each inner model of the form $\mathsf M_n(u)$ with $x\in\mathsf M_n(u)$ and Woodin cardinals $\delta_0<\cdots<\delta_{n-1}<\omega_1$, and each choice of indiscernibles $c_0<\cdots<c_{k-1}$ for $\mathsf M_n(u)$,
\[\psi(x) \Leftrightarrow \mathsf M_n(u)\models \varphi(x,\delta_0,\ldots,\delta_{n-1},c_0,\ldots,c_{k-1}).\]

\end{theorem}

As before, we are now able to deduce the following result:

\begin{theorem}\label{neemancons}
If $M_n^\sharp(a)$ exists for all $a\subseteq\lambda$, then $ \boldsymbol\Sigma^1_{n+3}\text{-}\mathrm{Coll}(\omega,\lambda)\text{-Absoluteness}$ holds.
\end{theorem}

\begin{proof}
The proof is analogue to the one for sharps (Theorem \ref{martincons}); with the main difference of applying Neeman's Theorem \ref{neeeeeeman} in place of Martin's Theorem \ref{martincountable}.
\end{proof}

Now we are ready to prove the generalisation of Theorem \ref{main-1}.

\begin{theorem}\label{nge5}
Assume $M_n^\sharp(a)$ exists for some $n\ge0$ and all $a\subseteq\lambda$. Then,  \[\boldsymbol\Sigma^1_{n+4}\text{-}\mathrm{Coll}(\omega,\lambda)\text{-Absoluteness} \Rightarrow \boldsymbol\Sigma^1_{n+4}\text{-}\mathbb P\text{-Absoluteness}\] for $\lambda$ a cardinal and $\mathbb P$ a forcing notion with $|\mathbb P|\le\lambda$.    
\end{theorem}

\begin{proof}We are going to prove our statement by induction. The case $n=0$ is given by Theorem \ref{main-1}. Thus, let $n>0$, and assume the result holds for $0\le n'<n$.\medskip


Let now $g\subseteq\mathbb P$ be $\mathsf V$-generic, $H\subseteq\mathrm{Coll}(\omega,\lambda)$ be $\mathsf V[g]$-generic, and $G\subseteq\mathrm{Coll}(\omega,\lambda)$ be $\mathsf V$-generic and such that $\mathsf V[g][H]=\mathsf V[G]$, thanks to Kripke's Theorem \ref{terminal}.\medskip

By inductive hypothesis and Lemma  \ref{upwards},  we only need to show downward absoluteness of $\boldsymbol\Sigma^1_{n+4}$-sentences from $\mathsf V[g]$ to $\mathsf V$. By Lemma \ref{abs4}, we also have
\[\mathsf V[g] \models \forall a\subseteq\lambda\ ``M_n^\sharp(a)\ \text{exists}".\]
Thus, by Theorem \ref{neemancons}, we know that \[\mathsf V[g]\models\boldsymbol\Sigma^1_{n+3}\text{-}\mathrm{Coll}(\omega,\lambda)\text{-Absoluteness.}\]
If $\varphi$ is a $\boldsymbol\Sigma^1_{n+4}$-sentence true in $\mathsf V[g]$ with parameters in $\mathsf V$, by Lemma \ref{upwards} and since we know $\mathsf V[g]$ satisfies $\boldsymbol\Sigma^1_{n+3}$-$\mathrm{Coll}(\omega,\lambda)$-Absoluteness, we get $\mathsf V[g][H]=\mathsf V[G]\models\varphi$. But $\mathsf V$ satisfies $\boldsymbol\Sigma^1_{n+4}$-$\mathrm{Coll}(\omega,\lambda)$-Absoluteness, hence $\mathsf V\models\varphi$, as desired.
\end{proof}

We would now like to remove the hypothesis of the existence of $M_n^\sharp(a)$, as we did in Theorem \ref{n=4}. The first step would be a generalisation of Lemma \ref{abs1}, which is probably possible under hypotheses that guarantee the absoluteness of the model $\mathsf M_n(a)$ between forcing extensions. The second step, though, would be too much: indeed, as for Hauser-Woodin's theorem projective absoluteness is equiconsistent with the existence of strong cardinals, while the existence of mice would entail the consistency of Woodin cardinals, which are strictly above strongs in the large cardinal hierarchy. So there is no hope of obtaining a result analogous to \ref{n=4} for $n\ge5$ using these techniques. Another way may be possible, though.



\section*{Acknowledgements}

The author is a member of the Gruppo Nazionale per le Strutture Algebriche, Geometriche e le loro Applicazioni (GNSAGA) of the Istituto Nazionale di Alta
Matematica (INdAM). This article is part of his PhD thesis, written under the supervision of Joan Bagaria at the Universities of Barcelona and Trento.\medskip

I would like to wholeheartedly thank my PhD advisor, Joan Bagaria, for the many insightful discussions on this topic and for his careful reading of this paper. His guidance, suggestions, and support have been invaluable.\medskip

I would also like to thank Lukas Koschat, who first asked me Question \ref{lukas} and provided the initial input that led me to begin thinking about its solution.

\end{document}